\documentclass[a4paper]{amsart}
\usepackage{amssymb}
\usepackage{amscd}
\usepackage{amsthm}
\usepackage{amsmath}
\usepackage{latexsym}
\usepackage[all]{xy}
\usepackage[utf8]{inputenc}
\usepackage{enumitem}

\setlist[enumerate]{label={(\roman*)}}

\theoremstyle{plain}
\newtheorem{theorem}{Theorem}

\newtheorem{lemma}[theorem]{Lemma}

\theoremstyle{definition}
\newtheorem{definition}[theorem]{Definition}
\newtheorem{example}[theorem]{Example}

\theoremstyle{remark}
\newtheorem{notation}[theorem]{Notation}
\newtheorem*{remark}{Remark}
\newtheorem*{acknowledgment}{Acknowledgment}
\numberwithin{theorem}{section} 

\newcommand{\card}{\mbox{\rm{card\,}}}
\newcommand{\Add}{\mbox{\rm{Add\,}}}
\newcommand{\add}{\mbox{\rm{add\,}}}

\newcommand{\Filt}[1]{\operatorname{Filt}(#1)}

\newcommand{\cf}[1]{\mbox{\rm{cf}}(#1)}

\newcommand{\Hom}[3]{\operatorname{Hom}_{#1}(#2,#3)}
\newcommand{\Ext}[4]{\operatorname{Ext}^{#1}_{#2}(#3,#4)}

\newcommand{\rfmod}[1]{\mbox{\rm{mod}--}{#1}}
\newcommand{\rmod}[1]{\mbox{\rm{Mod}--}{#1}}

\newcommand{\Ker}[1]{\mbox{\rm{Ker}}(#1)}

\begin{document}

\title{Tree modules and limits of the approximation theory}
\author{Jan Trlifaj}
\address{Charles University, Faculty of Mathematics and Physics, Department of Algebra \\
Sokolovsk\'{a} 83, 186 75 Prague 8, Czech Republic}
\email{trlifaj@karlin.mff.cuni.cz}

\date{\today}
\subjclass[2010]{Primary: 16G70, 18G25. Secondary: 03E75, 05C05, 16D70, 16D90, 16E30.}%
\keywords{almost split sequences, approximations of modules, infinite dimensional tilting theory, set-theoretic homological algebra, tree modules.}
\thanks{Research supported by GA\v CR 17-23112S}
\begin{abstract} In this expository paper, we present a construction of tree modules and combine it with (infinite dimensional) tilting theory and relative Mittag-Leffler conditions in order to explore limits of the approximation theory of modules. We also present a recent generalization of this construction due to \v Saroch which applies to factorization properties of maps, and yields a solution of an old problem by Auslander concerning existence of almost split sequences.      
\end{abstract}

\dedicatory{}

\maketitle

\section*{Introduction} 

Relative homological algebra has gained momentum over the past two decades by the discovery of a number of new classes suitable for approximations (precovers and preenvelopes) by Enochs et al.\ \cite{EJ1}, \cite{GT2}, and by the discovery of direct connections between complete cotorsion pairs and model category structures on complexes of modules by Hovey \cite{EJ2}, \cite{Ho}. The rich supply of approximations makes it possible to solve various particular problems of module theory by employing fitting approximations, while the model category structures enable us to compute morphisms between objects of the unbounded derived category of modules as morphisms between their cofibrant and fibrant replacements modulo chain homotopy. 

Classes providing for approximations appeared ubiquitous in the early 2000s. A~first warning sign came in the form of an independence result by Eklof and Shelah \cite{ES2}: they proved that whether the class of all Whitehead groups is precovering or not depends on the extension of ZFC that we work in. Later on, the class of all $\aleph_1$-free groups was shown not to be precovering in ZFC (cf.\ \cite[\S5]{EGPT}). Since $\aleph_1$-free groups coincide with the flat Mittag-Leffler ones, the focus moved to Mitag-Leffler conditions and their relative versions defining various classes of locally free modules. Relative Mittag-Leffler conditions were known to be connected to (infinite dimensional) tilting theory by the work of Angeleri and Herbera \cite{AH}. However, even the basic case of (absolute) Mittag-Leffler modules resisted, with gradual improvements obtained in \cite{BS} and \cite{SaT}. 

The recent complete solution for the general case involving relative Mittag-Leffler conditions in \cite{AST} and \cite{Sa2} has revealed a surprising fact: despite the complexity of the set-theoretic homological algebra proofs, the modules witnessing non-existence of approximations can always be taken small, meaning countably generated. Their construction in the absolute case goes back to a classic work of Bass \cite{Ba}, where countably generated flat non-projective modules were constructed over any non-right perfect ring. The witnessing countably generated modules in the general setting have therefore been called \emph{Bass modules}. 

Tilting theory turned out to be helpful in passing from the absolute to the relative cases in \cite{AST} and \cite{Sa2}. If $T$ is an (infinitely generated) tilting module, then the corresponding relative Mittag-Leffler modules are called \emph{locally $T$-free}. Their class turns out to be precovering, iff $T$ is $\Sigma$-pure split, that is, iff each pure submodule of a module in $\Add (T)$ splits. For example, each $\Sigma$-pure-injective tilting module is $\Sigma$-pure split, and hence so are all finitely generated tilting modules over any artin algebra. However, already for hereditary artin algebras of infinite representation type, there do exist infinitely generated tilting modules $T$ that are not $\Sigma$-pure split \cite{AKT}. So the non-precovering phenomenon for locally $T$-free modules does occur even in the artin algebra setting. 

\medskip
The key technical tool in proving non-existence of approximations (or factorization properties) of modules are the \emph{tree modules}. Their construction, and its generalizations and applications, are the main topic of this expository paper. In Section~\ref{Ch1}, we define Bass modules and the tree modules induced by them first in the the basic (absolute) case, and then in the general setting. We present their principal properties, and the limits that they impose on the approximation theory of modules. This part is based on \cite{AST}, \cite{Sa2}, and \cite{ST1}.  
  
In Section~\ref{Ch2}, we present a recent generalization of the tree module construction due to \v Saroch \cite{Sa1}, used to answer a 40 years old question by Auslander on existence of almost split sequences from \cite{A1}. We will discuss his proof of the fact that for every associative unital ring $R$, a module $M$ is the right term in an almost split sequence in $\rmod R$, iff $M$ is finitely presented and has local endomorphism ring.  
                                                           
\medskip
In what follows, $R$ will denote an associative (but not necessarily commutative) ring with unit, and $\rmod R$ the category of all (unitary right $R$-) modules. Moreover, given an infinite cardinal $\kappa$ and a module $M$, we call $M$ \emph{$\kappa$-presented} (\emph{$< \kappa$-presented}) provided that $M$ is the direct limit of a direct system $\mathcal D$ of cardinality $\leq \kappa$ ($< \kappa$) such that $\mathcal D$ consists of finitely presented modules. For a class of modules $\mathcal C$, we will use the notation $\mathcal C ^{<\kappa}$ and $\mathcal C ^{\leq\kappa}$ to denote the subclass of $\mathcal C$ consisting of all $< \kappa$-presented modules, and $\kappa$-presented modules, respectively. 

The notation $\rfmod R$ stands for the class of all \emph{strongly finitely presented} modules, i.e, the modules possessing a projective resolution consisting of finitely generated projective modules. For example, if $R$ is a right coherent ring, then $\rfmod R = (\rmod R)^{<\omega}$ is the class of all finitely presented modules.

\section{Tree modules and their applications}\label{Ch1}

\subsection{Tree modules}

The tree modules that we consider here were originally developed, in various mutations, for the setting of abelian groups by Eklof, Mekler and Shelah in order to study almost free groups with various additional properties (e.g., the $\aleph_1$-separable and the Whitehead ones, cf.\ \cite[XVII.2]{EM} and \cite{ES1}). Our version here is based on the tree $T_\kappa$ of all finite sequences of ordinals less than a given infinite cardinal $\kappa$. This version was employed e.g.\ in \cite{GT1}. Our tree modules arise by a uniform decoration of the trees $T_\kappa$ with Bass modules. In order to explain the basic idea in more detail, we need to introduce further notation. 

Let $\kappa$ be an infinite cardinal. The nodes of the \emph{tree $T_\kappa$} are the maps (or sequences) of the form $\tau : n \to \kappa$ where $n < \omega$. Here, $n = \ell(\tau)$ is the {\it length} of $\tau$. The partial order $\leq$ on $T_\kappa$ is defined by $\tau ^\prime \leq \tau$, if  $\ell(\tau ^\prime) \leq \ell(\tau)$ and $\tau \restriction \ell(\tau ^\prime) = \tau ^\prime$. In other words, $\tau ^\prime \leq \tau$, if the sequence $\tau ^\prime$ forms an initial segment of the sequence $\tau$. 

For example, $\tau ^\prime = (0, 3, 1) \leq \tau = (0, 3, 1, 0)$, but $\tau^\prime = (1) \nleq \tau = (3,3)$. Notice that the tree $T_\kappa$ has cardinality $\kappa$, and each node of $T_\kappa$ has $\kappa$ immediate successors. 

However, the branches of $T_\kappa$ are short: they correspond 1-1 to countable sequences of ordinals $< \kappa$, that is, to the maps $\nu : \omega \to \kappa$. In particular, $\hbox{Br}(T_\kappa)$, the set of all branches of $T_\kappa$, has cardinality $\kappa^\omega$.   

\begin{remark}\label{rem_card_arit} Of course, $\kappa \leq \kappa^\omega$. We will especially be interested in the cardinals $\kappa$ such that $\kappa < \kappa^\omega = 2^\kappa$. It is easy to see that there is a very good supply -- a proper class -- of such cardinals: for each cardinal $\mu$, there is $\kappa \geq \mu$ such that $\kappa^\omega = 2^\kappa$ (see e.g.\ \cite[8.26]{GT2}).
\end{remark}    

\medskip
\begin{notation}\label{notat_bas} Let $\mathcal F$ be a set of countably presented modules. A {\it Bass module} for $\mathcal F$ is a module $B$ which is a direct limit of a countably infinite direct system of modules from $\mathcal F$. Possibly taking a cofinal subset, we can w.l.o.g. express $B$ as a direct limit $B = \varinjlim_{i < \omega}  F_i$ where $\mathfrak D = ( F_i, f_{ji} \mid i \leq j < \omega )$ is a direct system of modules from $\mathcal F$ indexed by $\omega$. This is the direct system that will be used for decoration of branches of the tree $T_\kappa$. 
\end{notation}

The following particular example of a Bass module is a model one. Later on, we will see that the decoration of $T_\kappa$ by this module always yields a flat Mittag-Leffler module: 

\begin{example}\label{Bass_ex} Let $\mathcal F = \{ R \}$, $a_i \in R$, and $f_i : R \to R$ be the left multiplication by $a_i$ for each $i < \omega$. Then the chain $R \overset{f_0}\to R \overset{f_1}\to \dots R \overset{f_i} \to R \overset{f_{i+1}}\to \dots$ yields a direct system $\mathcal D$ (with $f_{ji} = f_{j,j-1} ... f_{i+1,i}$ for all $i < j < \omega$) whose direct limit is the countably presented module $B$ with the presentation $0 \to R^{(\omega )} \overset{f}\to R^{(\omega )} \to B \to 0$  where $f$ maps the $i$th term $1_i$ of the canonical basis of $R^{(\omega )}$ to $1_i - 1_{i+1}.a_i$. This is the \emph{classic Bass module}.

By \cite{Ba}, $B$ is a flat module of projective dimension at most $1$, and if $B$ is projective, then the chain of principal left ideals 
$$Ra_0 \supseteq Ra_1a_0 \supseteq \dots Ra_n...a_0 \supseteq Ra_{n+1}a_n...a_0 \supseteq ...$$ 
stabilizes. In particular, if $R$ is not right perfect, then there exists a classic Bass module which is not projective.  
\end{example}
   
\begin{remark} If $\mathcal F$ is any class of countably generated projective modules, then the Bass modules over $\mathcal F$ are flat and countably presented, and hence of projective dimension at most $1$. Conversely, each countably presented flat module is a Bass module for some class of finitely generated free modules, see e.g.\ \cite[2.23]{GT2}. 
\end{remark}
  
\medskip
\begin{notation}\label{notat_tree} Now we turn to the construction of the tree modules. The idea is to combine the facts that all branches of the tree $T_\kappa$ have length $\omega$ and the Bass module $B$ is a direct limit of the direct system $\mathfrak D = ( F_i, f_{ji} \mid i \leq j < \omega )$ of modules indexed in $\omega$, and uniformly decorate the branches of $T_\kappa$ with the members of $\mathcal D$. The resulting tree module $L$ is defined as a submodule of the product $P = \prod_{\tau \in T_\kappa} F_{\ell(\tau)}$ as follows:

For each $\nu \in \hbox{Br}(T_\kappa)$, $i < \omega$, and $x \in F_i$, we define $x_{\nu i} \in P$ by 
$$\pi_{\nu \restriction i} (x_{\nu i}) = x,$$ 
$$\pi_{\nu \restriction j} (x_{\nu i}) = f_{ji}(x) \hbox{ for each } i < j < \omega,$$ 
$$\pi_\tau (x_{\nu i}) = 0 \hbox{ otherwise},$$
where $\pi_\tau \in \Hom {R}{P}{F_{\ell(\tau)}}$ denotes the $\tau$th projection for each $\tau \in T_\kappa$.

Let $X_{\nu i} = \{ x_{\nu i} \mid x \in F_i \}$. Then $X_{\nu i}$ is a submodule of $P$ isomorphic to $F_i$. Further, let $X_\nu = \sum_{i < \omega} X_{\nu i}$, and $L = \sum_{\nu \in \hbox{Br}(T_\kappa)} X_\nu$. The module $L$ is the \emph{tree module} corresponding to $\kappa$ and to the presentation of $B$ as the direct limit of the direct system $\mathcal D$. 
\end{notation}

Here is our first observation concerning the module $L$:

\begin{lemma}\label{tree_module_L} 
\begin{enumerate}
\item $X_\nu \cong \bigoplus_{i < \omega} F_i$ for each $\nu \in \mbox{Br}(T_\kappa)$.
\item Let $D = \bigoplus_{\tau \in T_\kappa} F_{\ell(\tau)}$. Then $D \subseteq L \subseteq P$, and we have the tree module exact sequence 
$$0 \to D \overset{\subseteq}\to L \to B^{(\hbox{Br}(T_\kappa))} \to 0.$$
\end{enumerate}  
\end{lemma}
\begin{proof} 1. Let $Y_{\nu i} = \sum_{j \leq i} X_{\nu i}$. Then $Y_{\nu i} = \bigoplus_{j < i} F_j \oplus X_{\nu i}$, and the inclusion $Y_{\nu i} \subseteq Y_{\nu, i+1}$ splits, because there is a split exact sequence 
$$0 \to X_{\nu i} \overset{p}\hookrightarrow F_i \oplus X_{\nu, i+1} \overset{q}\to F_{i+1} \to 0$$
where $p (x_{\nu i}) = x + (f_i(x))_{\nu, i+1}$, and $q (z + x_{\nu, i+1}) = x - f_i(z)$.  

2. Let $\nu \in  \mbox{Br}(T_\kappa)$. Then $B \cong (X_\nu + D)/D$. Indeed,  for each $i < \omega$, we can define $g_i : F_i \to (X_\nu + D)/D$ by $g_i(x) = x_{\nu i} + D$. Then $((X_\nu + D)/D, g_i \mid i \in I)$ is the direct limit of the direct system $\mathfrak D$.

Since each element of $X_\nu$ is a sequence in $P$ whose $\tau$th component is zero for all $\tau \notin \{ \nu \restriction i \mid i < \omega \}$, the modules $((X_\nu + D)/D \mid \nu \in \mbox{Br}(T_\kappa) )$ are independent. Thus $L/D = \bigoplus_{\nu \in  \mbox{Br}(T_\kappa)} (X_\nu + D)/D \cong B^{(\mbox{Br}(T_\kappa))}$.
\end{proof}

\subsection{Locally $\mathcal F$-free modules}

Tree modules can be used to construct locally free modules. Before introducing them, we recall several useful definitions.
  
\begin{definition}\label{filt} Let $\mathcal C$ be a class of modules. A module $M$ is \emph{$\mathcal C$-filtered} (or a \emph{transfinite extension} of the modules in $\mathcal C$), provided that there exists an increasing chain $\mathcal M = ( M_\alpha \mid \alpha \leq \sigma )$ of submodules of $M$ with the following properties: 

\begin{enumerate}
\item $M_0 = 0$, 
\item $M_\alpha = \bigcup_{\beta < \alpha} M_\beta$ for each limit ordinal $\alpha \leq \sigma$, 
\item $M_{\alpha +1}/M_\alpha \cong C_\alpha$ for some $C_\alpha \in \mathcal C$, and 
\item $M_\sigma = M$. 
\end{enumerate} 

The chain $\mathcal M$ is a \emph{$\mathcal C$-filtration} (of length $\sigma$) of the module $M$. The class of all $\mathcal C$-filtered modules will be denoted by $\Filt {\mathcal C}$.

A class $\mathcal C \subseteq \rmod R$ is \emph{deconstructible} provided there exists a subset $\mathcal C ^\prime \subseteq \mathcal C$ such that $\mathcal C = \Filt {\mathcal C ^\prime}$. If moreover $\mathcal C ^\prime$ can be chosen so that $\mathcal C ^\prime \subseteq \mathcal C ^{<\kappa}$ for an infinite cardinal $\kappa$, then $\mathcal C$ is called \emph{$\kappa$-deconstructible}. 
\end{definition}

Transfinite extensions include direct sums and extensions:

\begin{example}\label{cases} 1. Let $M = \bigoplus_{\alpha < \sigma} C_\alpha$ be a direct sums of copies of modules from $\mathcal C$. Then $M$ is $\mathcal C$-filtered (just take $M_\alpha = \bigoplus_{\beta < \alpha} C_\beta$ for each $\alpha \leq \sigma$).
 
2. Let $0 \to C_1 \to M \to C_2 \to 0$ be an extension of the modules $C_1, C_2 \in \mathcal C$. Then $M$ is $\mathcal C$-filtered (take $M_0 = 0$, $M_1 = C_1$, and $M_2 = M$).       
\end{example}

The class $\Filt {\mathcal C}$ is obviously closed under transfinite extensions (i.e., $\Filt {\Filt {\mathcal C}}  = \Filt {\mathcal C$}), and hence under extensions and (arbitrary) direct sums. 

However, $\Filt {\mathcal C}$ is usually much larger than the closure of $\mathcal C$ under extensions and direct sums. For example, if $\mathcal C$ is the class of all simple modules, then the latter closure is just the class of all semisimple modules while $\Filt {\mathcal C}$ is the class of all semiartinian modules. 

\medskip
Deconstructible classes provide for approximations, and make it possible to do relative homological algebra. 
We postpone discussing them in more detail after introducing basics of the approximation theory later in this Section. 

\medskip
We will now turn to locally free modules. We recall the notation of \cite{HT} and \cite{ST1}:

\begin{definition}\label{dense} Let $R$ be a ring, $M$ a module, and $\kappa$ an infinite regular cardinal. 

A system $\mathcal S$ consisting of $<\kappa$-presented submodules of $M$ satisfying the conditions 
\begin{enumerate}
\item $\mathcal S$ is closed under unions of well-ordered ascending chains of length $<\kappa$, and
\item each subset $X \subseteq M$ such that $\card X < \kappa$ is contained in some $N \in \mathcal S$,
\end{enumerate}
is called  a \emph{$\kappa$-dense system} of submodules of $M$.
\end{definition}

Notice that in the setting of Definition \ref{dense}, $M$ is the directed union of the modules from $\mathcal S$. 

In order to connect this notion with the setting of \ref{notat_bas}, we consider $\kappa$-dense systems consisting of countably $\mathcal F$-filtered modules:

\begin{definition}\label{locally_F-free} Let $\mathcal F$ be a set of countably presented modules. Denote by $\mathcal C$ the class of all modules possessing a countable $\mathcal F$-filtration. Let $M$ be a module.
 
Then $M$ is \emph{locally $\mathcal F$-free} provided that $M$ contains an $\aleph _1$-dense system $\mathcal S$ consisting of submodules from $\mathcal C$. The system $\mathcal S$ is said to \emph{witness} the locally $\mathcal F$-freeness of $M$.  
\end{definition}

Notice that if $M$ is countably generated, then $M$ is locally $\mathcal F$-free, iff $M \in \mathcal C$. We also note the following result from \cite{ST1}:

\begin{lemma}\label{closed_under} Let $\mathcal F$ be a set of countably presented modules. Then the class of all locally $\mathcal F$-free modules is closed under transfinite extensions. In particular, it contains $\Filt {\mathcal F}$.
\end{lemma}  

\medskip
Let us pause to have another look at our basic setting:

\begin{example} Let $\mathcal F$ be the class of all countably presented projective modules. Then $\mathcal C = \mathcal F$, and $\mathcal F$-filtered modules coincide the projective modules (by a classic result of Kaplansky).

The locally $\mathcal F$-free modules are called \emph{$\aleph_1$-projective}, \cite{EM}. They are characterized by the existence of an $\aleph_1$-dense system consisting of countably generated projective modules. 
\end{example}

$\aleph_1$-projective modules are nothing else than the well-known \emph{flat Mittag-Leffler modules}, i.e., the flat modules $M$ such that for each family $( N_i \mid i \in I )$ of left $R$-modules, the canonical map $M \otimes \prod_{i \in I} N_i \to \prod_{i \in I} (M \otimes_R N_i)$ is monic. In fact, we have (see \cite{HT} and \cite{RG}):

\begin{lemma}\label{mitlef} The following are equivalent for a module $M$:
\begin{enumerate}
\item $M$ is $\aleph_1$-projective;
\item $M$ is flat Mittag-Leffler;
\item For some (or any) presentation $M = \varinjlim_{i \in I} F_i$ of $M$ as a direct limit of a direct system of finitely presented modules $( M_i \mid i \in I )$, and for each module $N$, the inverse system $( \Hom R{M_i}N \mid i \in I )$ has the Mittag-Leffler property.\end{enumerate}
\end{lemma} 

Recall that an inverse system of modules $\mathcal H = (H_i, h_{ij} \mid i \leq j \in I )$ has the \emph{Mittag-Leffler property}, if for each $k \in I$ there exists $k \leq j \in I$, such that $\mbox{Im}(h_{kj}) = \mbox{Im}(h_{ki})$ for each $j \leq i \in I$, that is, the terms of the decreasing chain $( \mbox{Im}(h_{ki}) \mid k \leq i \in I )$ of submodules of $H_k$ stabilize. 

\medskip
We will denote by $\mathcal{FM}$ the class of all flat Mittag-Leffler modules. If $R$ is a right perfect ring, then $\mathcal P _0 = \mathcal{FM} = \mathcal F _0$ \cite{Ba}, where $\mathcal P _0$ and $\mathcal F _0$ denotes the class of all projective and flat modules, respectively. However, if $R$ is not right perfect, then $\mathcal P _0 \subsetneq \mathcal{FM} \subsetneq \mathcal F _0$, and we will see later that these three classes have rather different structural properties.  

\medskip
Now, we can prove that our general construction of the tree module $L$ always yields a locally $\mathcal F$-free module. 
Of course, if $B$ is $\mathcal F$-filtered, then so is $L$. But if $B$ is not $\mathcal F$-filtered, the result is quite surprising: $L$ has only a {\lq}small{\rq} $\mathcal F$-filtered submodule $D$, while the {\lq}big{\rq} quotient $L/D$ is a direct sum of copies of $B$ (cf.\ Lemma \ref{tree_module_L}). 

\begin{lemma}\label{L} The tree module $L$ from Lemma \ref{tree_module_L} is locally $\mathcal F$-free.
\end{lemma}
\begin{proof} For each countable subset $C = \{ \nu_i \mid i < \omega \}$ of $\mbox{Br}(T_\kappa)$, the module $X_C = \sum_{\nu \in C} X_\nu$ 
is isomorphic to a countable direct sum of the $F_i$s. In fact, $X_C = \bigcup_{i < \omega} X_{C_i}$ where 
$X_{C_i} = \sum_{j \leq i} X_{\nu_j}$ is a direct summand in $X_{C_{i+1}}$ with a complement isomorphic to a countable direct sum of the $F_i$s. So the local $\mathcal F$-freeness of $L$ is witnessed by the set $\mathcal S$ of all $X_C$, where $C$ runs over all countable subsets of $\mbox{Br}(T_\kappa)$. 
\end{proof}

\subsection{Module approximations}

In this section, we recall the relevant basic notions and results from the approximation theory of modules. For more details we refer to  \cite[Part II]{GT2}. 

Approximations of modules were introduced by Auslander, Reiten and Smal{\o} in the setting of finitely generated modules over artin algebras while Enochs and Xu studied them in the general setting of $\rmod R$, albeit using the different terminology of precovers and preenvelopes. We will primarily be interested in the general setting, so our terminology follows \cite{EJ1}:  

\begin{definition} 
\begin{itemize}
\item[\rm{(i)}] A class of modules $\mathcal A$ is \emph{precovering} 
if for each module $M$ there is $f \in \mbox{Hom}_R(A,M)$ with $A \in \mathcal A$ such that
each $f^\prime \in \mbox{Hom}_R(A^{\prime},M)$ with $A^\prime \in \mathcal A$ has a factorization through 
$f$: 
\[
\xymatrix{A \ar[r]^{f} & M \\ 
{A^\prime} \ar@{-->}[u]^{g} \ar[ur]_{f^\prime} &}
\]
The map $f$ is called an \emph{$\mathcal A$-precover} of $M$.
\item[\rm{(ii)}] An $\mathcal A$-precover is \emph{special} in case it is surjective, and its kernel $K$ satisfies 
$\mbox{Ext}_R^1(A,K) = 0$ for each $A \in \mathcal A$. 
\item[\rm{(iii)}] Let $\mathcal A$ be precovering. Assume that in the setting of (i), if $f^\prime = f$ then each factorization $g$ is an automorphism. Then $f$ is called an \emph{$\mathcal A$-cover} of $M$. The class $\mathcal A$ is \emph{covering} in case each module has an $\mathcal A$-cover. 
\end{itemize}

Dually, we define \emph{(special) preenveloping} and \emph{enveloping} classes of modules. 
\end{definition}

Precovering classes are ubiquitous because of the following basic facts due to Enochs and \v S\v tov\'\i\v cek:

\begin{theorem}\label{abundance_precov} Let $\mathcal C$ be a class of modules. Then the following implications hold:  
\begin{enumerate}
\item If $\mathcal C$ is deconstructible, then $\mathcal C$ is precovering. 
\item If $\mathcal C$ is precovering and closed under direct limits, then $\mathcal C$ is covering.
\end{enumerate}
\end{theorem}

We note that though verified in many particular instances, the validity of the converse implications in both statements of Theorem \ref{abundance_precov} remains an open problem in general.

\medskip
One can obtain special precovering and special preenveloping classes by employing the notion of a complete cotorsion pair due to Salce \cite{S}:

\begin{definition}\label{Salce}
A pair of classes of modules $\mathfrak C = (\mathcal A, \mathcal B)$ is a \emph{cotorsion pair} provided that  
\begin{enumerate}
\item $\mathcal A = {}^\perp \mathcal B = \{ A \in \mbox{Mod-}R \mid \mbox{Ext}^1_R(A,B) = 0 \mbox{ for all } B \in \mathcal B \}$, and
\item $\mathcal B = \mathcal A ^\perp = \{ B \in \mbox{Mod-}R \mid \mbox{Ext}^1_R(A,B) = 0 \mbox{ for all } A \in \mathcal A \}$. 
\end{enumerate}
In this case $\mathcal A$ is closed under transfinite extensions. If moreover 

(3) For each module $M$, there exists an exact sequences $0 \to M \to B \to A \to 0$ with $A \in \mathcal A$ and $B \in \mathcal B$, 

then $\mathfrak C$ is called \emph{complete}. It that case, for each module $M$ there also exists an exact sequences $0 \to B^\prime \to A^\prime \to M \to 0$ with $A^\prime \in \mathcal A$ and $B^\prime \in \mathcal B$, whence $\mathcal A$ is a special precovering class and $\mathcal B$ a special preenveloping class.
\end{definition}

The ubiquity of special approximations comes from the following (cf.\ \cite{ET})

\begin{theorem}\label{abundance_special}
Let $\mathcal S$ be any set of modules. Then the cotorsion pair $(^\perp(\mathcal S ^\perp),\mathcal S ^\perp)$ is complete. In fact, for each module $M$ there exists an exact sequence $0 \to M \to B \to A \to 0$ where $B \in \mathcal S ^\perp$ and $A$ is $\mathcal S$-filtered.

Moreover, if $R \in \mathcal S$, then $\mathcal A = {}^\perp(\mathcal S ^\perp)$ coincides with the class of all direct summands of $\mathcal S$-filtered modules, and $\mathcal A$ is deconstructible.   
\end{theorem}

A further tool for constructing precovering classes comes from \cite{ST1}:

\begin{theorem}\label{abundance_complex} Assume that $\kappa$ is an infinite cardinal such that each right ideal of $R$ is $\leq \kappa$-generated. Let $0 \leq n < \omega$ and $\mathcal C$ be any $\kappa^+$-deconstructible class of modules. Then the class of all modules possessing a $\mathcal C$-resolution of length $\leq n$ is also $\kappa^+$-deconstructible.         
\end{theorem}

\begin{example} Since each projective module is a direct sum of countably generated modules, if $0 \leq n < \omega$ and $\kappa$ is an infinite cardinal such that each right ideal of $R$ is $\leq \kappa$-generated, then the class $\mathcal P_n$ of all modules of projective dimension at most $n$ is $\kappa^+$-deconstructible. 

Similarly, the class $\mathcal F _0$ of all flat modules over any ring is $\kappa^+$-deconstructible where $\kappa$ is an infinite cardinal $\geq \card R$, cf.\ \cite{BEE}. Hence so is the class $\mathcal F _n$ of all modules of projective dimension at most $n$, for each $n < \omega$.    
\end{example}

\subsection{Limits of the approximation theory}

Theorems \ref{abundance_precov}, \ref{abundance_special} and \ref{abundance_complex} yield numerous approximation classes of modules. However, not all classes of modules closed under transfinite extensions are precovering. This surprising fact can be proved using locally free (tree) modules. 

\medskip
We will now present a full proof for the absolute case of $\aleph_1$-projective (= flat Mittag-Leffler) modules over any non-right perfect ring. The result was first proved in a different way in \cite{EGPT} for the particular case of $\aleph_1$-projective abelian groups. For countable non-right perfect rings, a proof was given in \cite{BS} (cf.\ also \cite{SaT} and \cite{ST1}). The following proof for general non-right perfect rings using tree modules is due to \v Saroch:  

\begin{theorem}\label{Sar_fml} Assume that $R$ is a non-right perfect ring and let $\mathcal F = \{ R \}$. Let $B$ be a Bass module for $\mathcal F$ which is not projective (see Example \ref{Bass_ex}). Then $B$ has no $\aleph_1$-projective precover. 

In particular, the class $\mathcal{FM}$ is not precovering, and hence not deconstructible.
\end{theorem}
\begin{proof} Assume there exists a $\mathcal{FM}$-precover $f : F \to B$ of $B$. Consider the short exact sequence $0 \to K \hookrightarrow	F \overset{f}\to B \to 0$ where $K = \mbox{Ker}(f)$. Let $\kappa$ be an infinite cardinal such that $\card R \leq\kappa$ and $\card K \leq 2^\kappa = \kappa ^\omega$ (see Remark \ref{rem_card_arit}). 

Consider the tree module corresponding to $T_\kappa$, $\mathcal F$ and $B$. By Lemma \ref{tree_module_L}, we have the tree module short exact sequence 
$$0 \to D \hookrightarrow L \to {B^{(2^\kappa)}} \to 0,$$ 
with $L \in \mathcal{FM}$ and $D$ a free module of rank $\kappa$. Clearly, $L\in\mathcal P _1$. 

\medskip
Let $\eta:K \hookrightarrow E$ be a $\{L\}^\perp$-preenvelope of $K$ with an $\{ L\}$-filtered cokernel $C$ (cf.\ Theorem \ref{abundance_special}). Consider the pushout
$$\begin{CD} 
    @. 0 @. 0 @. @. \\
    @. @VVV  @VVV @.   @. \\
	0 @>>> K @>{\subseteq}>>	F @>{f}>> B @>>> 0 	\\
	@. @V{\eta}VV			@V{\varepsilon}VV	  @|   @. \\
	0 @>>> E @>{\subseteq}>>	P @>{g}>> B  @>>> 0 \\
	@. @VVV			@VVV	  @.   @. \\
	  @. \mbox C @>{\cong}>> \mbox C @. @. \\
    @. @VVV  @VVV @.   @. \\
    @. 0 @. \hbox{ }0 @. @.
\end{CD}$$
Then $P\in\mathcal{FM}$. Since $f$ is an $\mathcal{FM}$-precover, there exists $h: P\to F$ such that $fh = g$. Then $f = g\varepsilon = fh\varepsilon$, whence $K+\mbox{Im}(h) = F$. Let $h^\prime = h\restriction E$. Then $h^\prime:E\to K$ and $\mbox{Im}(h^\prime) = K \cap \mbox{Im}(h)$.

Consider the restricted short exact sequence
$$\begin{CD} 0 @>>> \mbox{Im}(h^\prime) @>{\subseteq}>> \mbox{Im}(h) @>{f\restriction \mbox{Im}(h)}>> B @>>> 0.\end{CD}$$
As $E\in L^\perp$ and $L\in\mathcal P_1$, also $\mbox{Im}(h^\prime)\in L^\perp$. 

Applying $\Hom R{-}{\mbox{Im}(h^\prime)}$ to the tree-module short exact sequence above, we obtain the exact sequence 
$$\Hom RD{\mbox{Im}(h^\prime)} \to \Ext 1RB{\mbox{Im}(h^\prime)}^{2^\kappa}\to 0$$
where the first term has cardinality at most $(\card K)^\kappa \leq 2^\kappa$, so $\mbox{Im}(h^\prime)\in B^\perp$ (otherwise, the second term would have cardinality at least $2^{2^\kappa}$).  

Then $f\restriction \mbox{Im}(h)$ splits, and so does the $\mathcal{FM}$-precover $f$. It follows that $B \in \mathcal{FM}$, whence $B$ is projective, a contradiction.
\end{proof}

Theorem \ref{Sar_fml} is a special instance (for $\mathcal F = \{ R \}$) of the following more general result proved in \cite{Sa2}. We will present its applications in the following section. 

\begin{theorem}\label{Sar_gen}  Let $\mathcal F$ be a class of countably presented modules, and $\mathcal L$ the class of all locally $\mathcal F$-free modules. Let $B$ be a Bass module for $\mathcal F$ such that $B$ is not a direct summand in a module from $\mathcal L$. 

Then $B$ has no $\mathcal L$-precover. In particular, the class $\mathcal L$ is not precovering, and hence not deconstructible. 
\end{theorem}
   
\subsection{Tilting approximations and locally $T$-free modules}

The model case of flat Mittag-Leffler modules discussed in the previous section is actually a $0$-dimensional instance of a more general phenomenon related to tilting. This relation was first noticed in \cite{ST1}, and fully developed and generalized in \cite{AST}. 

In order to formulate the key results, we recall the notion of an (infinitely generated) tilting module. For more details, we refer e.g.\ to \cite[Part III]{GT2}. 

For a module $T$, denote by $\mbox{Add}(T)$ (resp.\ $\mbox{add}(T))$ the class of all direct summands of arbitrary (resp.\ finite) direct sums of copies of $T$. 

\begin{definition}\label{tilt} A module $T$ is \emph{tilting} provided that
\begin{itemize}
\item[\rm{(T1)}] $T$ has finite projective dimension,
\item[\rm{(T2)}] $\mbox{Ext}^i_R(T,T^{(\kappa)}) = 0$ for all $1 \leq i < \omega$ and all cardinals $\kappa$, and
\item[\rm{(T3)}] there exist $r < \omega$ and an exact sequence $0 \to R \to T_0  \to \dots \to T_r \to 0$ where $T_i \in \Add T$ for each $i \leq r$.
\end{itemize}

\noindent The class 
$$\mathcal T _T = T^{\perp_\infty} = \{ B \in \mbox{Mod-}R \mid \mbox{Ext}^i_R(T,B) = 0 \mbox{ for all } 0 < i < \omega \}$$ is called the \emph{(right) tilting class}, $\mathcal A _T = {}^\perp \mathcal T _T$ the \emph{left tilting class}, and the (complete) cotorsion pair $\mathfrak C _T = (\mathcal A _T, \mathcal T _T)$ the \emph{tilting cotorsion pair}, induced by $T$.  

If $T$ has projective dimension $\leq n$, then the tilting module $T$ is called \emph{$n$-tilting}, and similarly for $\mathcal T _T$, $\mathcal A _T$, and $\mathfrak C _T$. 
\end{definition}

It is easy to see that $0$-tilting modules $T$ coincide with (not necessarily finitely generated) projective generators, whence $\mathcal A _T = \mathcal P _0$ and $\mathcal T _T = \rmod R$. 

Also, for each tilting module $T$, we have $\Add T = \mathcal A _T \cap \mathcal B _T$ (this is the \emph{kernel} of the tilting cotorsion pair $\mathfrak C _T$), and the right tilting class $\mathcal B _T$ is \emph{definable}, that is, $\mathcal B _T$ is closed under direct limits, products and pure submodules.   

\medskip
Tilting theory originated in the setting of finitely generated modules over finite dimensional algebras, but many of its aspects extend to the general setting of modules over arbitrary rings. Such extension is especially interesting for commutative rings, because all finitely generated tilting modules over a commutative ring are projective. For a recent classification of tilting classes over commutative rings, we refer to \cite{APST} and \cite{HS}. 

Though the left tilting class $\mathcal A _T$ is always special precovering and the right tilting class $\mathcal T _T$ special preenveloping, one can employ tilting modules to produce non-precovering classes of modules, namely the classes of locally $T$-free modules. The construction generalizes the base case of $T = R$, where locally $T$-free modules coincide with the flat Mittag-Leffler ones.

Before explaining the construction in more detail, we recall basic facts of infinite dimensional tilting theory over arbitrary rings:     

\begin{theorem}\label{chart} 
\begin{enumerate}
\item Let $\mathfrak C = (\mathcal A,\mathcal B)$ be a cotorsion pair. Then $\mathfrak C$ is tilting, iff $\mathcal A \subseteq \mathcal P _n$ for some $n < \omega$, and $\mathcal B$ is closed under arbitrary direct sums. 
\item Right tilting classes $\mathcal T$ in $\rmod R$ coincide with the classes of finite type, i.e., the classes of the form $\mathcal S ^{\perp_\infty}$ where $\mathcal S$ consists of strongly finitely presented modules of bounded projective dimension. Such class $\mathcal T$ is $n$-tilting, iff $\mathcal S \subseteq \mathcal P _n$.    

In particular, each left tilting class is $\aleph_1$-deconstructible.     
\end{enumerate}
\end{theorem} 

The largest possible choice for the class $\mathcal S$ in Theorem \ref{chart}.2 is $\mathcal S = \mathcal A \cap \rfmod R$ where $\mathcal A = {}^\perp \mathcal T$. Then $\mathcal S \subseteq \mathcal A \subseteq \varinjlim \mathcal S$ (where $\varinjlim \mathcal S$ denotes the class of all direct limits of modules from $\mathcal S$). 

In the $0$-tilting case, the largest choice is $\mathcal S = \mathcal P _0 ^{< \omega}$ (the class of all finitely generated projective modules), whence $\varinjlim \mathcal S = \mathcal F _0$ (the class of all flat modules). Of course, we also have $\mathcal P _0 \subseteq \mathcal{FM} \subseteq \mathcal F _0$, and the question is how to generalize the notion of a flat Mittag-Leffler module to the $n$-tilting setting for $n > 0$. The answer is given by the following definition:
   
\begin{definition}\label{locallytf} Consider the particular setting of Notation \ref{notat_tree} when $T$ is a tilting module and $\mathcal F=\mathcal A _T^{\leq \omega}$ (so $\mathcal C= \mathcal F$). 

A module is \emph{locally $T$-free} provided that $M$ is locally $\mathcal F$-free, i.e., $M$ admits a dense system of countably presented submodules from $\mathcal A$. We will denote by $\mathcal L _T$ the class of all locally $T$-free modules.
\end{definition}

So if $T = R$, then the locally $T$-free modules coincide with the $\aleph_1$-projective (= flat Mittag-Leffler) modules. By Theorem \ref{Sar_fml}, these modules form a (pre-) covering class, iff $R$ is a right perfect ring, iff $\mathcal P _0 = \mathcal F _0$. This was substantially generalized in \cite{AST} as follows: 

\begin{theorem}\label{general} Let $T$ be a tilting module and $\mathcal A _T$ be the corresponding left tilting class 
(so $\mathcal A _T = \Filt{(\mathcal C _T)}$ where $\mathcal C _T = \mathcal A _T ^{\leq \omega}$ by Theorem \ref{chart}.2). Let $\mathcal L _T$ denote the class of all locally $T$-free modules. Then the following are equivalent:
\begin{enumerate}
\item $\mathcal L _T$ is a (pre-) covering class;
\item Each Bass module for $\mathcal C _T$ is contained in $\mathcal C _T$;
\item The class $\mathcal A_T$ is closed under direct limits;
\item $T$ is $\sum$-pure split. 
\end{enumerate}
\end{theorem} 

Here, a module $T$ is \emph{$\sum$-pure split} provided that each pure embedding $N \hookrightarrow M$ with $M \in \Add (T)$, splits. For example, any $\sum$-pure injective module is $\sum$-pure split. Also, a ring $R$ is right perfect, iff the regular module is $\sum$-pure split.

Since $\mathcal S _T = \mathcal A _T ^{< \omega} \subseteq \mathcal A _T \subseteq \mathcal L _T \subseteq \varinjlim \mathcal S _T$, condition (3) above is further equivalent to $\mathcal L _T$ being closed under direct limits. 

\medskip
Theorem \ref{general} indicates that the phenomenon of existence of non-precovering classes closed under transfinite extensions is much more widespread than originally expected. For example, though each finitely generated module over an artin algebra is $\sum$-pure injective (because it is endofinite), there do exist countably generated non-$\sum$-pure split tilting modules over each hereditary artin algebra of infinite representation type:

\begin{example}\label{Lukas} Let $R$ be an indecomposable hereditary finite dimensional algebra of infinite representation type. Recall that there is a partition of the represetative set of all indecomposable finitely generated modules, $\mbox{ind-}R$, into three sets:
$q$ - the indecomposable preinjective modules, $p$ - the indecomposable preprojective modules, and $t$ the indecomposable regular modules.

By Theorem \ref{chart}.2, $p^\perp$ is a right tilting class. The tilting module $T$ inducing $p^\perp$ is called the \emph{Lukas tilting module}. The left tilting class of $T$ is the class $\mathcal R$ of all \emph{Baer modules}. By \cite{AKT}, $\mathcal R = \Filt {p}$. The locally $T$-free modules are called \emph{locally Baer modules}. 

By \cite{AKT} the Lukas tilting module is countably generated, but it has no finitely generated direct summands, and it is not $\Sigma$-pure split. Therefore, by Theorem \ref{general}, the class of all locally Baer modules is not precovering (and hence not deconstructible).

Of course, this means that there exist Bass modules for $\mathcal R ^{< \omega}$ that are not Baer. Since $\mathcal R ^{< \omega} = \mbox{add}(p)$ is the class of all finitely generated preprojective modules, these Bass modules can be obtained as unions of the chains 
$$P_0 \overset{f_0}\hookrightarrow P_1 \overset{f_1}\hookrightarrow \dots \overset{f_{i-1}}\hookrightarrow P_i \overset{f_i}\hookrightarrow P_{i+1} \overset{f_{i+1}}\hookrightarrow \dots$$
such that all the $P_i$ are finitely generated preprojective, but the cokernels of all the $f_i$ are regular (i.e., in $\add (t)$). Such chains exist for any hereditary finite dimensional algebra of infinite representation type, see \cite{AST}.    
\end{example}

Theorem \ref{general} is proved in \cite{AST} as a corollary of a still more general result concerning cotorsion pairs:

\begin{theorem}\label{more_general} Let $\mathfrak C = (\mathcal A, \mathcal B)$ be a cotorsion pair such that the class $\mathcal B$ is closed under direct limits. Then $\mathcal A$ is $\aleph_1$-deconstructible (whence $\mathfrak C$ is complete), and there is a module $T$ such that $\mathcal A \cap \mathcal B = \Add T$. 

Let $\mathcal F = \mathcal A ^{\leq \omega}$ and $\mathcal L$ be the class of all locally $\mathcal F$-free modules. Then the following are equivalent:
\begin{enumerate}
\item $\mathcal L$ is a (pre-) covering class;
\item Each Bass module for $\mathcal F$ is contained in $\mathcal F$;
\item The class $\mathcal A$ is closed under direct limits;
\item $K$ is $\sum$-pure split. 
\end{enumerate}
\end{theorem} 

The proof of Theorem \ref{more_general} employs relative Mittag-Leffler conditions studied in \cite{AH}, \cite{H} and \cite{R}. One of its interesting by-products is an alternative description of local $\mathcal F$-freeness using these conditions in the general setting of Theorem \ref{more_general}:  

\begin{definition} Let $\mathcal B$ be a class of modules. A module $M$ is \emph{$\mathcal B$-stationary} provided that $M$ can be expressed as the direct limit of a direct system $\mathcal M$ of finitely presented modules such that for each $B \in \mathcal B$, the induced inverse system $\mathcal H$ is Mittag-Leffler (see Lemma \ref{mitlef}). 
\end{definition}

\begin{theorem}\label{rel_stat} Let $M \in \rmod R$ and $(\mathcal A,\mathcal B)$ be a cotorsion pair in $\rmod R$ such that $\mathcal B$ is closed under direct limits. The the following are equivalent:
\begin{enumerate}
\item $M$ is locally $\mathcal A^{\leq\omega}$-free; 
\item $M$ is $\mathcal B$-stationary, and $M$ is a pure-epimorphic image of a module from $\mathcal A$.   
\end{enumerate}
In particular, if $M \in \varinjlim \mathcal A$, then $M$ is locally $\mathcal A^{\leq\omega}$-free, iff $M$ is $\mathcal B$-stationary. 
\end{theorem}
          
\begin{remark}\label{still_more_general} Lemma \ref{Sar_gen} is one of the key tools for proving Theorem \ref{more_general}. However, it can be used to construct non-precovering classes even outside the setting of \ref{more_general}: in \cite{ST2}, inspired by the ideas of Positselski from algebraic geometry \cite{P}, the authors studied very flat and locally very flat modules over commutative rings $R$. In analogy with our basic case of projective and flat Mittag-Leffler modules, they proved that for each noetherian domain, the class of all very flat modules is covering, iff the class of all locally very flat modules is precovering, iff the Zariski spectrum Spec(R) is finite.   
\end{remark}

\section{Tree modules and almost split sequences}\label{Ch2}

Almost split sequences play a central role in the representation theory of artin algebras. Since the founding work of Auslander and Reiten \cite{AR}, they provide a crucial tool for understanding the category of finitely generated modules, notably for studying extensions of indecomposable finitely generated modules. 

We will only briefly touch upon this vast topic. As in Section \ref{Ch1}, we will stick to the general setting of (possibly infinitely generated) modules over arbitrary rings. After introducing the necessary basic definitions in this setting, we will directly concentrate on a major problem formulated by Auslander in 1975 in \cite{A1}: Auslander asked whether the notion of a right almost split map in $\rmod R$ ending in a module $M$, available for any ring $R$ and any indecomposable finitely presented module $M$ with local endomorphism ring, does occur also for some infinitely generated modules. 
A negative solution has recently been obtained by \v Saroch in \cite{Sa1} and it uses (generalized) tree modules. Before explaining the solution in more detail (following \cite{Sa1}), we recall the relevant definitions:

\begin{definition}\label{almost}  Let $R$ be a ring. A morphism $f \in \Hom RBC$ is said to be \emph{right almost split} (in $\rmod R$) provided that the following two conditions are equivalent for each $B^\prime \in \rmod R$ and $h \in \Hom R{B^\prime}C$: 

\begin{enumerate}
\item $h$ factorizes through $f$;
\item $h$ is not a split epimorphism.
\end{enumerate}

\emph{Left almost split} morphisms are defined dually. A short exact sequence $0 \to A \overset{g}\to B \overset{f}\to C \to 0$ in $\rmod R$ is an \emph{almost split sequence} provided that $f$ is right almost split and $g$ is left almost split.  
\end{definition}

If $f$ is right almost split, then it is easy to see that $f$ is not a split epimorphism, and the endomorphism ring of $C$ is local (whence $C$ is indecomposable). Auslander \cite{A2} proved that the converse holds true when $C$ is finitely presented:

\begin{theorem}\label{ausl} Let $C$ be a finitely presented module. Then there exists a right almost split morphism $f \in \Hom RBC$ in $\rmod R$, iff the endomorphism ring of $C$ is local. If this is the case and $C$ is not projective, then there exists an almost split sequence  $0 \to A \to B \to C \to 0$ in $\rmod R$.
\end{theorem}

\begin{remark} In the particular case when $R$ is an artin algebra, much more can be proved: the almost split sequences are unique and form the (simple) socles of the extension modules. More precisely, if $C \in \rfmod R$ is indecomposable and non-projective, then there exists a unique (up to equivalence of short exact sequences) almost split sequence  $0 \to A \to B \to C \to 0$, and  $A \in \rfmod R$ is indecomposable non-injective. The equivalence class of this almost split sequence forms the socle of the left $S$-module $\Ext 1RCA$ (where $S$ is the endomorphism ring of $C$), and this socle is a simple left $S$-module. Moreover, also the dual statement holds true --- see e.g.\ \cite[V.2]{ARS} for more details.    
\end{remark}  

In \cite{A1}, Auslander asked whether given a ring $R$ and a module $C$, there exists a right (left) almost split morphism in $\rmod R$ ending (beginning) in $C$. The main result of \cite{Sa1} answers this question for right almost split morphisms:

\begin{theorem}\label{saroch} Let $R$ be a ring and $C \in \rmod R$. Then there exists a right almost split map $f \in \Hom RBC$ in $\rmod R$, iff $C$ is finitely presented and the endomorphism ring of $C$ is local.
\end{theorem}
  
\medskip
Now, we proceed to indicate the modifications of the construction of tree modules $T_\kappa$ from Section \ref{Ch1} that are needed for the proof of Theorem \ref{saroch}. 

First, instead of the Bass modules $B$, which are direct limits of the direct systems $\mathcal D = ( F_i, f_{ji} \mid i \leq j < \omega )$ of small (= countably presented) modules $F_i$ indexed in $\omega$, we will need to consider direct limits of general well-ordered continuous direct systems $\mathcal E = ( F_\alpha, f_{\beta \alpha} \mid \alpha \leq \beta < \cf{\theta}  )$ of small (= $< \theta$-presented) modules $F_\alpha$ indexed by the regular infinite cardinal $\cf{\theta}$. Here, $\cf{\theta}$ denotes the cofinality of $\theta$, and the term \emph{continuous} means that for each limit ordinal $\alpha < \cf{\theta}$, we have $F_\alpha = \varinjlim_{\beta < \alpha} F_\beta$. 

These general direct systems are important because of the following classic result by Iwamura:

\begin{lemma}\label{sar_lem} Let $R$ be a ring, $\theta$ be an infinite cardinal, and $C$ a $\theta$-presented module. Then $C$ is the direct limit of a well-ordered continuous direct systems $\mathcal E = ( F_\alpha, f_{\beta \alpha} \mid \alpha \leq \beta < \cf{\theta} )$ such that $F_\alpha$ is a $< \theta$-presented module for each $\alpha < \cf{\theta}$.
\end{lemma}         
 
The second non-trivial modification concerns the trees $T_\kappa$. The modified trees will again be uniformly decorated (by the well-ordered continuous direct system $\mathcal E$). But we now have two initial parameters: $\mu$ and $\kappa$, where $\mu = \cf{\theta}$ is a regular infinite cardinal, and $\kappa$ is an infinite cardinal such that $\kappa ^{<\mu} = \kappa$. 

The latter equality is used in \cite{Sa1} to define the modified tree $T$ so that the set of all its branches $\hbox{Br}(T)$ is a certain subset of $\kappa ^\mu$ of cardinality $\kappa ^\mu$ such that any two distinct elements of $\hbox{Br}(T)$ coincide on an initial segment of $\mu$. The partial order on $T \subseteq \mu \times \kappa$ is defined by $(\alpha,\beta) \leq (\gamma,\delta)$, if $\alpha < \gamma$ and there exists $\eta \in \hbox{Br}(T)$ such that $\eta(\alpha) = \beta$ and $\eta(\gamma) = \delta$. 

Such construction of the tree $T$ has the advantage that it still enables a presentation of the resulting tree module as a submodule of a product of copies of the modules $F_\alpha$ ($\alpha < \mu$). In order to explain the latter fact, we recall the following lemma relating well-ordered direct limits to (reduced) products (see e.g.\ \cite[3.3.2]{Pr}):

\begin{lemma}\label{prest} Let $\mu$ be a regular infinite cardinal and $\mathcal E = ( F_\alpha, f_{\beta \alpha} \mid \alpha \leq \beta < \mu )$ be a well-ordered continuous direct system of modules. Then there is the following commutative diagram  
$$\begin{CD}
0@>>>  K@>{\subseteq}>> D@>{\pi}>> E@>>>  0\\
@.     @VVV    @V{\tau_0}VV @V{\sigma_0}VV @.\\
0@>>>  Q@>{\subseteq}>> P@>>>   {P/Q}@>>>   0
\end{CD}$$
where $D = {\bigoplus_{\alpha < \mu} F_\alpha}$, $E = \varinjlim_{\alpha < \mu} F_\alpha$, $\pi : D \to E$ is the canonical pure epimorphism, $K = \Ker {\pi}$, $P = {\prod_{\alpha < \mu} F_\alpha}$, $Q$ is the submodule of $P$ consisting of the sequences with support of cardinality $< \mu$, and $\tau_0$ is defined as follows: for each $\alpha < \mu$ and $x \in F_\alpha \subseteq D$, the $\beta$th component of the sequence $\tau_0(x)$ equals $f_{\beta \alpha}(x)$ in case $\alpha \leq \beta < \mu$, and it is $0$ otherwise. Moreover, the induced map $\sigma _0$ is a pure monomorphism.   
\end{lemma}   

\begin{remark} While the first row of the diagram in Lemma \ref{prest} is the usual pure-exact sequence presenting a direct limit as the quotient of a direct sum, the monomorphism $\tau_0 : D \to P$ is not the usual pure embedding of the direct sum into the direct product. In fact, in the particular case of $\mu = \omega$, $\tau_0 \restriction F_i$ maps $F_i$ into $P$ the same way as the assignment $x \to x_{\nu i}$ in Notation \ref{notat_tree}.
\end{remark}   

Now, we will uniformly decorate the tree $T$ by the direct system $\mathcal E$. First we use the tree $T$ to have a multidimensional version of the commutative diagram from Lemma \ref{prest} as follows:
$$\begin{CD}
0@>>>  K^{(\hbox{Br}(T))}@>{\subseteq}>> D^{(\hbox{Br}(T))}@>>> E^{(\hbox{Br}(T))}@>>>  0\\
@.     @VVV    @V{\tau}VV @V{\sigma}VV @.\\
0@>>>  {Q^\prime}@>>> P^\prime@>{\rho}>>   {P^\prime/Q^\prime}@>>>   0
\end{CD}$$
where $P^\prime = \prod_{(\alpha,\beta) \in T} F_{\alpha \beta}$, $Q^\prime$ is the submodule of $P^\prime$ consisting of the sequences with support of cardinality $< \mu$, and $F_{\alpha \beta} = F_\alpha$ for all $(\alpha,\beta) \in T$. 

Here, the upper exact sequence is just the direct sum of $\kappa ^\mu$ copies of the canonical presentation of $E$ as a pure-epimorphic image of $D$ (i.e., of the first row of the diagram in Lemma \ref{prest}), and $\tau$ restricted to the $\eta$th component (for $\eta \in \hbox{Br}(T)$) acts as $\tau_0$ from Lemma \ref{prest} (again, compare with Notation \ref{notat_tree}).       

Finally, we can define the generalized tree module:

\begin{lemma}\label{key} In the setting above, let $L = \hbox{Im}(\tau)$ and $H = \Ker {\rho \restriction L}$. Then we have the tree module exact sequence
$$0 \to H \to L \overset{g}\to C^{(\hbox{Br}(T))} \to 0$$ 
where $C = \varinjlim F_\alpha$ is the direct limit of the direct system $\mathcal E$.    
\end{lemma}

Following the pattern of Section \ref{Ch1}, we should now proceed to showing that, in some sense, $L$ is an almost free module:

\begin{definition} Let $\theta$ be an infinite cardinal and $M \in \rmod R$. Then $M$ is \emph{finitely $\theta$-separable} provided that $M$ is the directed union of a direct system consisting of $< \theta$-presented direct summands of $M$.
\end{definition}

\begin{lemma}\label{splite} The module $L$ from Lemma \ref{key} is finitely $\theta$-separable. Moreover, if $e$ is an idempotent endomorphism of $L$ such that $e \restriction H = 0$ and the endomorphism ring of $\hbox{Im}{(e)}$ is local, then $\hbox{Im}{(e)}$ is $< \theta$-presented.  
\end{lemma}   

The tree module from Lemma \ref{tree_module_L} was used in the proof of Theorem \ref{Sar_fml} to show that the Bass module $B$ has no $\aleph_1$-projective precover. Our goal now is to use the tree module from Lemma \ref{key} to show that almost split morphism cannot terminate in an infinitely generated module. The main tool is the following lemma from \cite{Sa1}:

\begin{lemma}\label{L-test}  Let $\theta$ be an infinite cardinal, $C$ be an $\theta$-presented module and $f : B \to C$ be a non-split epimorphism. Let $\kappa$ be an infinite cardinal such that $\card{R}, \theta \leq \kappa$, $\kappa = \kappa ^{< \cf{\theta}}$, and $\card{\Ker f} \leq \kappa ^{\cf{\theta}} = 2^\kappa$. 

Let $L$ be the tree module from Lemma \ref{key}. Then there exists $h \in \Hom R{C^{(\hbox{Br}(T))}}C$ such that $hg$ does not factorize through $f$.      
\end{lemma} 

With all these tools at hand, we can prove Theorem \ref{saroch}: 
\begin{proof} The if-part follows from Auslander's Theorem \ref{ausl}. For the only-if part, assume there exists an infinite cardinal $\theta$ and a right almost split morphism $f : B \to C$ such that $C$ is a $\theta$-presented, but not $< \theta$-presented module. Since $C$ is not finitely presented, each homomorphism from a finitely presented module to $C$ factorizes through $f$, whence $f$ is a (non-split) pure epimorphism. Moreover, the endomorphism ring of $C$ is local. 

The tree module $L$ from Lemma \ref{L-test} has the property that there exists $h \in \Hom R{C^{(\hbox{Br}(T))}}C$ such that $hg$ does not factorize through $f$. Then $hg : L \to C$ must be a split epimorphism, so there is $\iota : C \to L$ such that $h g \iota = 1_C$. Using Lemma \ref{splite} for the idempotent endomorphism $e = \iota h g$, we infer that $C \cong \hbox{Im}(e)$ is $< \theta$-presented, a contradiction.     
\end{proof}
  
\begin{remark} As noted in \cite{Sa1}, Theorem \ref{saroch} imposes strict limitations on the form of almost split sequences in $\rmod R$. Namely, it implies that if $0 \to A \overset{h}\to B \to C \to 0$ is an almost split sequence in $\rmod R$, then not only $C$ is finitely presented with local endomorphism ring, but also $A$ is pure-injective: otherwise, the non-split pure-embedding of $A$ into its pure-injective envelope factors through $h$, whence $h$ is a pure monomorphism and the almost split sequence splits, a contradiction.
\end{remark}

\begin{acknowledgment} The author thanks Jan \v Saroch for his comments on the first draft of this paper. 
\end{acknowledgment}


\begin{thebibliography}{EGPT}

\bibitem{AH}
{L.\ Angeleri H\"{u}gel, D.\ Herbera}, \textit{Mittag-Leffler conditions on modules}, Indiana  Math.\ J. 57(2008), 2459-2517.

\bibitem{AKT}
{L.\ Angeleri H\" ugel, O.\ Kerner, J.\ Trlifaj}, \textit{Large tilting modules and representation type}, Manuscripta Math. 132(2010), 483-499. 

\bibitem{APST} {L.\ Angeleri H\" ugel, D.\ Posp\'\i\v sil, J.\ \v S\v{t}ov\'\i\v cek, J.\ Trlifaj},
\textit{Tilting, cotilting, and spectra of commutative noetherian rings}, Trans.\ Amer.\ Math.\ Soc. 366(2014), 3487-3517.
 
\bibitem{AST}
L.\ Angeleri H\"{u}gel, J.\ \v{S}aroch, J.\ Trlifaj, \textit{Approximations and Mittag-Leffler conditions -- the applications}, to appear in Israel J.\ Math., arXiv:1612.01140v1.

\bibitem{A1} {M.\ Auslander}, \textit{Existence theorems for almost split sequences}, in Ring theory II (Proc.\ Second Conf., Univ. Oklahoma, Norman, Okla., 1975), LNPAM 26, M.\ Dekker, New York 1977, 1-44.

\bibitem{A2} {M.\ Auslander}, \textit{A survey of existence theorems for almost split sequences}, in Repres. of Algebras (Proc.\ Conf.\ Durham 1985), LMSLNS 116, Cambridge Univ.\ Pres., Cambridge 1986, 81–89.

\bibitem{AR} {M.\ Auslander, I.\ Reiten}, \textit{Representation theory of artin algebras III. Almost split sequences}, Comm. Algebra 3(1975), 239-294.

\bibitem{ARS} {M.\ Auslander, I.\ Reiten, S.\ Smal\o}, \textit{Representation Theory of Artin Algebras}, CSAM 36, Cambridge Univ.\ Press, Cambridge 1995.

\bibitem{Ba}
H.\ Bass, \textit{Finitistic dimension and a homological generalization of semiprimary rings}, 
Trans.\ Amer.\ Math.\ Soc. 95(1960), 466-488.

\bibitem{BS} S.\ Bazzoni, J.\ \v S\v tov\'\i\v cek, \textit{Flat Mittag-Leffler modules over countable rings}, Proc.\ Amer.\ Math.\ Soc. 140(2012), 1527-1533.

\bibitem{BEE} L.\ Bican, R.\ El Bashir, E.\ E.\ Enochs, \textit{All modules have flat covers}, Bull.\ London Math.\ Soc. 33(2001), 385-390.

\bibitem{EM}
P.\ C.\ Eklof, A.\ H.\ Mekler, \textit{Almost Free Modules (Set-theoretic Methods)}, 2nd revised ed., North-Holland Math. Library, Elsevier, Amsterdam 2002.

\bibitem{ES1} 
P.\ C.\ Eklof, S.\ Shelah, \textit{On Whitehead modules}, J.\ Algebra 142(1991), 492-510.

\bibitem{ES2} 
P.\ C.\ Eklof, S.\ Shelah, \textit{On the existence of precovers}, Illinois J.\ Math. 47(2003), 173-188.

\bibitem{ET}
P.\ C.\ Eklof, J.\ Trlifaj, \textit{How to make Ext vanish}, Bull.\ London Math.\ Soc. 33(2001), 41-51.

\bibitem{EJ1}
E.\ E.\ Enochs, O.\ M.\ G.\ Jenda, \textit{Relative Homological Algebra 1}, 2nd rev.\ ext.\ ed., GEM 30, W.\ de Gruyter, Berlin 2011.

\bibitem{EJ2}
E.\ E.\ Enochs, O.\ M.\ G.\ Jenda, \textit{Relative Homological Algebra 2}, GEM 54, W.\ de Gruyter, Berlin 2011.

\bibitem{EGPT} S.\ Estrada, P.\ Guil Asensio, M.\ Prest, J.\ Trlifaj: \textit{Model category structures arising from Drinfeld vector bundles}, Adv.\ Math. 231(2012), 1417-1438. 

\bibitem{GT1}
R.\ G\"{o}bel, J.\ Trlifaj, \textit{Cotilting and a hierarchy of almost cotorsion groups}, J.\ Algebra 224(2000), 110-122.

\bibitem{GT2}
R.\ G\"{o}bel, J.\ Trlifaj, \textit{Approximations and Endomorphism Algebras of Modules}, 2nd rev.\ ext.\ ed., Vols.\ 1 and 2, GEM 41, W.\ de Gruyter, Berlin 2012.

\bibitem{H} D.\ Herbera, \textit{Definable classes and Mittag-Leffler conditions}, Ring Theory and Its Appl., Contemporary Math. 609(2014), 137-166. 

\bibitem{HT}
D.\ Herbera, J.\ Trlifaj, \textit{Almost free modules and Mittag-Leffler conditions}, Adv.\ Math. 229(2012), 3436-3467.

\bibitem{Ho}
M.\ Hovey, \textit{Cotorsion pairs, model category structures, and representation theory}, Math.\ Z. 241(2002), 553-592.

\bibitem{HS}
{M.\ Hrbek, J.\ \v{S}\v{t}ov\'{i}\v{c}ek}, \textit{Tilting classes over commutative rings}, preprint, arXiv:1701.05534v1. 

\bibitem{P} 
L.Positselski, \textit{Contraherent cosheaves}, preprint, arXiv:1209.2995v5.

\bibitem{Pr}
{M.\ Prest}, \textit{Purity, Spectra and Localisation},
Enc.\ Math.\ Appl. 121, Cambridge Univ.\ Press, Cambridge 2009.

\bibitem{RG}
{M.\ Raynaud, L.\ Gruson}, \textit{Crit\`eres de platitude et de projectivit\'e}, Invent.\ math.\ 13(1971), 1-89.

\bibitem{R} {P.\ Rothmaler}, \textit{Mittag-Leffler modules and positive atomicity}, Habilitationsschrift, Christian-Albrechts-Universit\"at zu Kiel 1994. 

\bibitem{S}
{L. Salce}, \textit{Cotorsion theories for abelian groups}, Sympos.\ Math. 23(1979), 11-32.

\bibitem{Sa1}
J.\ \v{S}aroch, \textit{On the non-existence of right almost split maps}, Invent.\ math. 209(2017), 463-479.

\bibitem{Sa2}
J.\ \v{S}aroch, \textit{Approximations and Mittag-Leffler conditions -- the tools}, to appear in Israel J.\ Math., arXiv:1612.01138v1.

\bibitem{SaT} J.\ \v Saroch, J.\ Trlifaj, \textit{Kaplansky classes, finite character, and $\aleph_1$-projectivity}, Forum Math. 24(2012), 1091-1109. 

\bibitem{ST1}
A.\ Sl\'{a}vik, J.\ Trlifaj, \textit{Approximations and locally free modules}, Bull.\ London Math.\ Soc. 46(2014), 76-90.

\bibitem{ST2}
A.\ Sl\'{a}vik, J.\ Trlifaj, \textit{Very flat, locally very flat, and contraadjusted modules}, J.\ Pure Appl.\ Algebra 220(2016), 3910-3926.

\end{thebibliography}
\end{document}